\documentclass[12pt]{article}

\usepackage{amsmath,amsthm,amsfonts,amssymb,amscd,cite}
\usepackage{tikz}
\usepackage{color}

\newtheorem{theorem}{Theorem}[section]
\newtheorem{conjecture}[theorem]{Conjecture}
\newtheorem{corollary}[theorem]{Corollary}

\newtheorem{lemma}[theorem]{Lemma}
\newtheorem{problem}[theorem]{Problem}

\begin{document}

\title{Interplay between the local metric dimension and the clique number of a graph}

\author{
Ali Ghalavand$^{a,}$\thanks{Corresponding author;  email:\texttt{alighalavand@nankai.edu.cn}}
\and
Sandi Klav\v zar $^{b,c,d}$
\and 
Xueliang Li$^{a}$
}

\maketitle

\begin{center}
$^a$ Center for Combinatorics, Nankai University, Tianjin 300071, China \\
\medskip

$^b$ Faculty of Mathematics and Physics,  University of Ljubljana, Slovenia\\
\medskip

$^b$ Institute of Mathematics, Physics and Mechanics, Ljubljana, Slovenia\\
\medskip

$^d$ Faculty of Natural Sciences and Mathematics,  University of Maribor, Slovenia\\
\end{center}

\begin{abstract}
The local metric dimension ${\rm dim}_l$ in relation to the clique number $\omega$ is investigated. It is proved that if $\omega(G)\leq n(G)-3$, then ${\rm dim}_l(G) \leq n(G)-3$ and the graphs attaining the bound classified. Moreover, the graphs $G$ with ${\rm dim}_l(G) = n(G)-3$ are listed (with no condition on the clique number). It is proved that if $\omega(G)=n(G)-2$, then $n(G)-4 \leq {\rm dim}_l(G)\leq n(G)-3$, and all graphs are divided into two groups depending on which of the options applies. The conjecture asserting that for any graph $G$ we have ${\rm dim}_l(G) \leq \left[(\omega(G)-2)/(\omega(G)-1)\right] \cdot n(G)$ is proved for all graphs with $\omega(G)\in\{n(G)-1,n(G)-2,n(G)-3\}$. A negative answer is given for the problem whether every planar graph fulfills the inequality ${\rm dim}_l(G) \leq \lceil (n(G)+1)/2 \rceil$. 
\end{abstract}


\medskip\noindent
\textbf{AMS Math.\ Subj.\ Class.\ (2020)}: 05C12, 05C69

\section{Introduction}

This paper focuses on studying the local metric dimension of finite, simple, and connected graphs. First, let's introduce this concept.

If $x$ and $y$ are vertices of a graph $G = (V(G), E(G))$, then the distance $d_G(x, y)$ between $x$ and $y$ in $G$ is the number of edges on a shortest $x,y$-path. Vertices $u$ and $v$ of $G$ are {\em distinguished} by $w$, or equivalently, $w$ {\em distinguishes} $u$ and $v$, if $d_G(u,w)\neq d_G(v,w)$. A subset $W\subseteq V(G)$ is a {\em resolving set} for $G$ if for any $u, v\in V(G)-W$, there is a vertex in $W$ that distinguishes $u$ and $v$. $W$ is a {\em local resolving set} if for any adjacent $u, v\in V(G)-W$, there is a vertex in $W$ that distinguishes $u$ and $v$. The {\em metric dimension} $\dim(G)$  and the {\em local metric dimension} $\dim_l(G)$ of $G$ are the cardinalities of smallest resolving sets and smallest local resolving sets for $G$, respectively. Clearly, $\dim_l(G)\leq \dim(G)$.

The concept of the metric dimension of graphs has a rich history dating back to its initial definition by Harary \& Melter~\cite{13} and Slater \cite{25}. Determining the metric dimension is NP-complete in general graphs~\cite{19} as well as restricted to planar graphs with a maximum degree 6~\cite{6}. Research on this dimension is numerous, partly due to the fact that it found various applications in real-world problems such as robot navigation, image processing, privacy in social networks, and locating intruders in networks. The 2023 overview~\cite{survey2} of the essential results and applications of metric dimension contains well over 200 references. 

Alongside research on the metric dimension, there have also been various reasons for considering variations of this concept. The survey~\cite{survey1} which focuses on variants of metric dimension also cites over 200 papers. One of the most interesting variations is the local metric dimension introduced in 2010 by, Okamoto, Phinezy, and Zhang~\cite{Okamoto1}. This dimension is also computationally hard~\cite{9,10} and has been investigated in a series of papers, let us point to the following selection of them~\cite{Abrishami1, 3, 4, Ghalavand1, lal-2023, 17, yang-2024}, as well as to the recent paper on the fractional local metric dimension~\cite{javaid-2024}. 

Denote by $\omega(G)$ the clique number of $G$, and by $n(G)$ the order of $G$. In the seminal paper on the local metric dimension, the following results were proved. 

\begin{theorem}\label{cth0}
{\rm \cite[Theorems 2.4, 2.5]{Okamoto1}} If $G$ is a connected graph with $n(G) \ge 3$, then the following hold. 
\begin{enumerate}
  \item[\rm (I)] $\dim_l(G)=n-1$ if and only if $G\cong\,K_n$.
  \item [\rm(II)] $\dim_l(G)=n-2$ if and only if $\omega(G)=n-1$.
  \item [\rm(III)] $\dim_l(G)=1$ if and only if  $G$ is bipartite.
\end{enumerate}
\end{theorem}

\begin{theorem}\label{thm:Okamoto-2}
{\rm \cite[Theorems 3.1, 3.2]{Okamoto1}} If $G$ is a connected graph, 
$$\dim_l(G) \ge \max\left \{ \lceil \log_2 \omega(G)\rceil, n(G) - 2^{n(G) - \omega(G)}\right\}\,.$$
\end{theorem}

The local metric dimension and the clique number of a graph are therefore strongly intertwined. This was the primary motivation for our present paper in which we extend the above two theorems as follows. 
\begin{align}
\omega(G) = n(G) & \Leftrightarrow \dim_l(G) = n(G)-1\,.\label{eq:1} \\
\omega(G) = n(G)-1 & \Leftrightarrow \dim_l(G) = n(G)-2\,.\label{eq:2} \\
\omega(G) = n(G)-2 & \Rightarrow n(G)-4 \le \dim_l(G) \le  n(G)-3\,.\label{eq:3} \\
\omega(G) = n(G)-3 & \Rightarrow n(G)-8 \le \dim_l(G) \le  n(G)-3\,.\label{eq:4}
\end{align}
These results are elaborated in Section~\ref{sec:omega}. We prove the upper bound of~\eqref{eq:4} in Theorem~\ref{th1} for the more general case when $\omega(G) \le n(G)-3$ holds, and also characterize the graphs which attain the bound. In Theorem~\ref{th2} we then characterize the graphs $G$ for which $\dim_l(G) = n(G)-3$. Applying the above theorems we then deduce the bounds in~\eqref{eq:3}. In Section~\ref{sec:conjectures} we first use the results of Section~\ref{sec:omega} to demonstrate that the conjecture asserting that for any graph $G$ we have ${\rm dim}_l(G) \leq \left[(\omega(G)-2)/(\omega(G)-1)\right] \cdot n(G)$ holds true for all graphs with $\omega(G)\in\{n(G)-1,n(G)-2,n(G)-3\}$. We end the section and the paper by showing that the answer to the problem whether for a planar graph we have $\dim_l(G) \leq \lceil (n(G)+1)/2 \rceil$, is negative. In the rest of the introduction we give further definitions needed and recall two results to be used later on. 

For a positive integer $t$, we denote the set $\{1,\ldots, t\}$ by $[t]$. For a subset $V'\subseteq V(G)$, the subgraph induced by $V'$ is denoted by $G[V']$. The open and the closed neighborhood of a vertex $u$ in $G$ are respectively denoted by $N_G(u)$ and $N_G[u]$. A clique $Q$ of a graph $G$ is a set of vertices that induce a complete subgraph of  $G$. By abuse of language we will also use the term clique for the subgraph induced by $Q$. Vertices $u$ and $v$ of a graph $G$ are {\em true twins} if $N_G[u] = N_G[v]$. Notice that true twins are adjacent vertices and that the relation of being a true twin is an equivalence relation on $V(G)$. The following bound will be  useful.

\begin{lemma} {\rm \cite[Observation 2.1]{Okamoto1}}
\label{lem:twins}
If $G$ is a connected graph having $k$ true twin equivalence classes, then $\dim_l(G) \ge n(G) - k$. 
\end{lemma}

In~\cite{isariyapalakul-2024} the role of true twin equivalence classes was investigated for the connected local dimension, that is, for the metric dimension where resolving sets are required to be connected. Finally, we recall the following result needed.  

\begin{theorem} {\rm \cite[Theorem 6]{Abrishami1}}
\label{cth1}
If $G$ is a triangle-free graph with $n(G) \geq 3$, then
\[\dim_l(G)\leq\frac{2}{5}n(G)\,.\]
\end{theorem}

\section{Local dimension when $\omega(G)\le n(G)-3$}
\label{sec:omega}

In this section we focus on the local metric dimension of graphs $G$ with $\omega(G)\in \{n(G)-3,n(G)-2\}$, as well as of $\omega(G)\le n(G) - 3$. 

Let $n\ge 3$ and let $\mu$ and $\lambda$ be positive integers with $1\le \mu\le \lambda$ and $\lambda + \mu \le n-1$. Then we will denote by $K_n^{-}(\lambda, \mu)$ the graph obtained from $K_n$ by removing the edges of some $K_{\lambda, \mu}$ subgraph. The structure of the obtained  graph does not depend on the selection of the subgraph $K_{\lambda, \mu}$, hence $K_n^{-}(\lambda, \mu)$ is well-defined. 

\begin{lemma}
\label{lem:Kn-}
If $n\ge 3$, $1\le \mu\le \lambda$, and $\lambda + \mu \le n-1$, then 
$$
\dim_l(K_n^{-}(\lambda, \mu)) = 
\begin{cases}
n-2; & \mu = 1, \\ 
n-3; & \text{otherwise}.
\end{cases}
$$
\end{lemma}

\begin{proof}
If $\mu = 1$, then $\omega(K_n^{-}(\lambda, \mu)) = n-1$, hence $\dim_l(K_n^{-}(\lambda, \mu)) = n-2$ by  Theorem~\ref{cth0}(II). 

Assume now that $\mu \ge 2$ (and $\mu\le \lambda$, $\lambda + \mu \le n-1$). Let $L$ and $M$ be the bipartition sets of a $K_{\lambda, \mu}$ subgraph, where $|L|=\lambda$, $|M|=\mu$, removing the edges of which produces $K_n^{-}(\lambda, \mu)$. Then $|L|\ge 2$, $|M|\ge 2$, and $\{L, M, V(K_n)\setminus (L\cup M)\}$ is a partition into twin-equivalence classes. Hence by Lemma~\ref{lem:twins}, $\dim_l(K_n^{-}(\lambda, \mu)) \ge n-3$. On the other hand, a set consisting of $|L|-1$ vertices from $L$, of $|M|-1$ vertices from $M$, and of $|V(K_n)\setminus (L\cup M)|-1$ vertices from $V(K_n)\setminus (L\cup M)$ is a local resolving set of cardinality $n-3$. 
\end{proof}

\begin{theorem}\label{th1}
If $G$ is a connected graph with $n(G) \geq 5$ and $\omega(G)\leq n(G)-3$, then $\dim_l(G) \leq n(G)-3$. Moreover, the equality holds if and only if $G$ is either $C_5$ or $K_n^{-}(\lambda, \mu)$, where $\lambda\geq\mu\geq2$ and $\lambda+\mu<n(G)$.
\end{theorem}

\begin{proof}
Set $n = n(G)$ and $k = \omega(G)$ for the rest of the proof. Then $n\ge 5$ and $k \le n - 3$. 

Let $Q$ be a clique of $G$ with $|V(Q)| = k$. Let $S$ be a largest subset of vertices of $Q$ such that for any vertices $s$ and $s'$ from $S$  we have $N_G(s)\cap(V(G)-Q) \neq N_G(s') \cap (V(G)-Q)$. Since $G$ is connected and $\omega(G)=k$, we infer that $|S| \geq 2$. Depending on the cardinality of $S$, we distinguish three cases. 

\medskip\noindent
{\bf Case 1}: $|S| \geq 4$. \\
By the definition of $S$, the set $V(G)-S$ is a local resolving set for $G$. Therefore, $\dim_l(G)\leq n-|S|\leq n-4$.

\medskip\noindent
{\bf Case 2}: $|S| = 3$. \\
Let $S=\{s_1,s_2,s_3\}$, and let $A$ be a minimal subset of $V(G)-S$ such that $N_G(s_i)\cap A\neq N_G(s_j)\cap A$ for each $i,j\in [3]$, $i\ne j$. Since $G[S]\cong K_3$, we infer that $|A|=2$. Let $A=\{a_1,a_2\}$ and let $v$ be an arbitrary vertex from the set $V(G)-(Q\cup A)$. Such a vertex exists since $\omega(G)\leq n(G)-3$.

If for each $i\in [3]$, $s_i$ is not adjacent to $v$, or  there is a vertex in $Q-S$ that is not adjacent to $v$, then $V(G)-\{s_1,s_2,s_3,v\}$ is a local resolving set for $G$ and we are done. 

If for each $i\in [3]$, the edge $vs_i$ exists, then since $\omega(G)=k$, at least one vertex in $Q-S$ is not adjacent to $v$. This means that the set $V(G)-\{s_1,s_2,s_3,v\}$ again forms a local resolving set for $G$. 

By the above we may now assume that $N_G(v)\cap\,(Q-S)=Q-S$ and $1\leq\,|N_G(v)\cap\,S|\leq2$. We consider two sub-cases, where we often employ the definitions of $S$ and $A$.

\medskip\noindent
{\bf Case 2.1}: $|N_G(v) \cap S| = 1$. \\
Let $i\in [3]$ be the index such that $vs_i\in E(G)$. 
If $N_G(s_i)\cap\,A\neq\,N_G(v)\cap\,A$, then $V(G)-\{s_1,s_2,s_3,v\}$ is a local resolving set for $G$. If $|N_G(s_i)\cap\,A|=|N_G(v)\cap\,A|=0$, then $V(G)-\{s_i,a_1,a_2,v\}$ is a local resolving set for $G$. If $N_G(s_i)\cap\,A=N_G(v)\cap\,A=\{a_1\}$, then $V(G)-\{s_i,s_j,a_2,v\}$ is a local resolving set for $G$, where $j\in[3]$ and $s_ja_1\not\in E(G)$. If $N_G(s_i)\cap\,A=N_G(v)\cap\,A=\{a_2\}$, then the proof is similar to last case and we omit it.  If $N_G(s_i)\cap\,A=N_G(v)\cap\,A=\{a_1,a_2\}$ and $|N_G(a_l)\cap\,S|=2$ for $l\in[2]$, then $V(G)-\{s_i,a_1,a_2,v\}$ is a local resolving set for $G$. If $N_G(s_i)\cap\,A=N_G(v)\cap\,A=\{a_1,a_2\}$ and $|N_G(a_l)\cap\,S|=1$ for $l\in[2]$, then $V(G)-\{s_i,s_j,a_l,v\}$ is a local resolving set for $G$, where $j\in[3]$ and $s_ja_l\not\in E(G)$.

\medskip\noindent
{\bf Case 2.2}: $|N_G(v) \cap S| = 2$. \\
If $|N_G(v) \cap A| = 0$, then $V(G) - \{s_l, a_1, a_2, v\}$ serves as a local resolving set for $G $, where among the vertices in $\{s_1, s_2, s_3\}$, the vertex $s_l$ has has the least number of neighbors in $\{a_1, a_2\}$. 

If $|N_G(v) \cap A| =1$ and for each $i\in[3]$, $|N_G(s_i)\cap\,A|\leq1$, then by applying the conditions on $A$ one can observe that there is an $l\in[3]$ such that $|N_G(s_l)\cap\,A|=0$ and $|N_G(s_h)\cap\,A|=1$ for $h\in([3]-\{l\})$. Therefore $V(G) - \{s_h, a_1, a_2, v\}$ is a local resolving set for $G$, where $h\in([3]-\{l\})$,  $N_G(s_h)\cap\,A=N_G(v)\cap\,A$. 

If $|N_G(v) \cap A| = 1$ and $\{|N_G(s_h)\cap A|:\ h\in[3]\}=\{0,1,2\}$, then by using the assumptions on both $S$ and $A$, one can verify that $G[S\cup A\cup\{v\}]$ is one of the graphs illustrated in Fig.~\ref{nfig1}. 

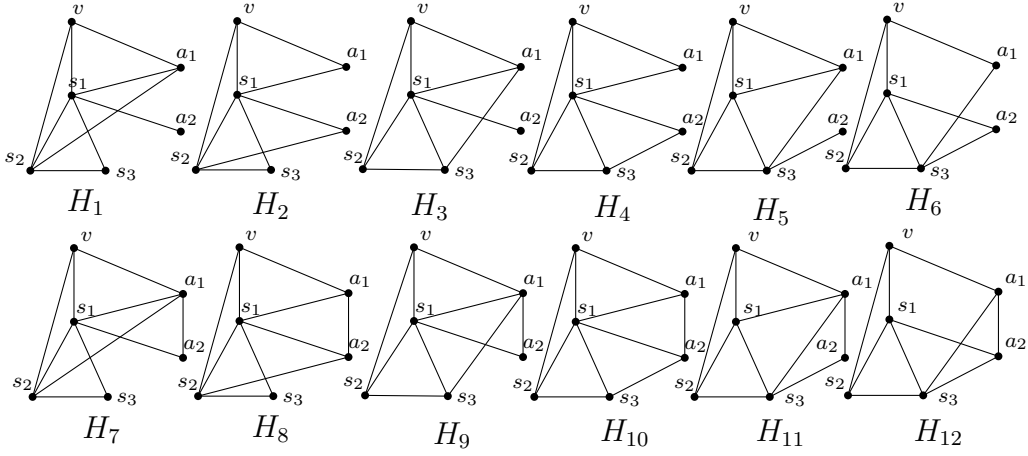
\begin{figure}[ht!]
\begin{center}
\begin{tikzpicture}
\clip(-6.5,-2.9) rectangle (7.5,3.6);
\draw (-6,1)-- (-5,1);
\draw (-5,1)-- (-5.45,2);
\draw (-5.45,2)-- (-6,1);
\draw (-5.45,2.98)-- (-6,1);
\draw (-5.45,2.98)-- (-5.45,2);
\draw (-5.45,2.98)-- (-4,2.37);
\draw (-5.45,2)-- (-4,2.37);
\draw (-5.45,2)-- (-4,1.52);
\draw (-6,1)-- (-4,2.37);
\draw (-3.8,1.01)-- (-2.8,1.01);
\draw (-2.8,1.01)-- (-3.25,2.01);
\draw (-3.25,2.01)-- (-3.8,1.01);
\draw (-3.25,2.99)-- (-3.8,1.01);
\draw (-3.25,2.99)-- (-3.25,2.01);
\draw (-3.25,2.99)-- (-1.8,2.38);
\draw (-3.25,2.01)-- (-1.8,2.38);
\draw (-3.25,2.01)-- (-1.8,1.53);
\draw (-3.8,1.01)-- (-1.8,1.53);
\draw (-1.58,1.02)-- (-0.49,1.01);
\draw (-0.49,1.01)-- (-0.94,2.01);
\draw (-0.94,2.01)-- (-1.58,1.02);
\draw (-0.94,2.99)-- (-1.58,1.02);
\draw (-0.94,2.99)-- (-0.94,2.01);
\draw (-0.94,2.99)-- (0.52,2.38);
\draw (-0.94,2.01)-- (0.52,2.38);
\draw (-0.94,2.01)-- (0.52,1.53);
\draw (-0.49,1.01)-- (0.52,2.38);
\draw (0.66,1)-- (1.66,1);
\draw (1.66,1)-- (1.21,2);
\draw (1.21,2)-- (0.66,1);
\draw (1.21,2.98)-- (0.66,1);
\draw (1.21,2.98)-- (1.21,2);
\draw (1.21,2.98)-- (2.67,2.37);
\draw (1.21,2)-- (2.67,2.37);
\draw (1.21,2)-- (2.67,1.52);
\draw (1.66,1)-- (2.67,1.52);
\draw (2.79,1)-- (3.79,1);
\draw (3.79,1)-- (3.34,2);
\draw (3.34,2)-- (2.79,1);
\draw (3.34,2.98)-- (2.79,1);
\draw (3.34,2.98)-- (3.34,2);
\draw (3.34,2.98)-- (4.8,2.37);
\draw (3.79,1)-- (4.8,2.37);
\draw (3.79,1)-- (4.8,1.52);
\draw (3.34,2)-- (4.8,2.37);
\draw (4.84,1.03)-- (5.84,1.03);
\draw (5.84,1.03)-- (5.39,2.03);
\draw (5.39,2.03)-- (4.84,1.03);
\draw (5.39,3.01)-- (4.84,1.03);
\draw (5.39,3.01)-- (5.39,2.03);
\draw (5.39,3.01)-- (6.84,2.4);
\draw (5.84,1.03)-- (6.84,1.55);
\draw (5.84,1.03)-- (6.84,2.4);
\draw (5.39,2.03)-- (6.84,1.55);
\draw (-5.97,-2.01)-- (-4.97,-2.01);
\draw (-4.97,-2.01)-- (-5.42,-1.01);
\draw (-5.42,-1.01)-- (-5.97,-2.01);
\draw (-5.42,-0.03)-- (-5.97,-2.01);
\draw (-5.42,-0.03)-- (-5.42,-1.01);
\draw (-5.42,-0.03)-- (-3.97,-0.64);
\draw (-5.42,-1.01)-- (-3.97,-0.64);
\draw (-5.42,-1.01)-- (-3.97,-1.49);
\draw (-5.97,-2.01)-- (-3.97,-0.64);
\draw (-3.77,-2)-- (-2.77,-2);
\draw (-2.77,-2)-- (-3.22,-1);
\draw (-3.22,-1)-- (-3.77,-2);
\draw (-3.22,-0.02)-- (-3.77,-2);
\draw (-3.22,-0.02)-- (-3.22,-1);
\draw (-3.22,-0.02)-- (-1.77,-0.63);
\draw (-3.22,-1)-- (-1.77,-0.63);
\draw (-3.22,-1)-- (-1.77,-1.48);
\draw (-3.77,-2)-- (-1.77,-1.48);
\draw (-1.55,-1.99)-- (-0.45,-2);
\draw (-0.45,-2)-- (-0.9,-1);
\draw (-0.9,-1)-- (-1.55,-1.99);
\draw (-0.9,-0.02)-- (-1.55,-1.99);
\draw (-0.9,-0.02)-- (-0.9,-1);
\draw (-0.9,-0.02)-- (0.55,-0.63);
\draw (-0.9,-1)-- (0.55,-0.63);
\draw (-0.9,-1)-- (0.55,-1.48);
\draw (-0.45,-2)-- (0.55,-0.63);
\draw (0.7,-2.01)-- (1.7,-2.01);
\draw (1.7,-2.01)-- (1.25,-1.01);
\draw (1.25,-1.01)-- (0.7,-2.01);
\draw (1.25,-0.03)-- (0.7,-2.01);
\draw (1.25,-0.03)-- (1.25,-1.01);
\draw (1.25,-0.03)-- (2.7,-0.64);
\draw (1.25,-1.01)-- (2.7,-0.64);
\draw (1.25,-1.01)-- (2.7,-1.49);
\draw (1.7,-2.01)-- (2.7,-1.49);
\draw (2.83,-2.01)-- (3.83,-2.01);
\draw (3.83,-2.01)-- (3.38,-1.01);
\draw (3.38,-1.01)-- (2.83,-2.01);
\draw (3.38,-0.03)-- (2.83,-2.01);
\draw (3.38,-0.03)-- (3.38,-1.01);
\draw (3.38,-0.03)-- (4.83,-0.64);
\draw (3.83,-2.01)-- (4.83,-0.64);
\draw (3.83,-2.01)-- (4.83,-1.49);
\draw (3.38,-1.01)-- (4.83,-0.64);
\draw (4.87,-1.99)-- (5.87,-1.99);
\draw (5.87,-1.99)-- (5.42,-0.98);
\draw (5.42,-0.98)-- (4.87,-1.99);
\draw (5.42,0)-- (4.87,-1.99);
\draw (5.42,0)-- (5.42,-0.98);
\draw (5.42,0)-- (6.87,-0.61);
\draw (5.87,-1.99)-- (6.87,-1.47);
\draw (5.87,-1.99)-- (6.87,-0.61);
\draw (5.42,-0.98)-- (6.87,-1.47);
\draw (-5.66,0.9) node[anchor=north west] {$H_1$};
\draw (-3.2,0.85) node[anchor=north west] {$H_2$};
\draw (-1.09,0.84) node[anchor=north west] {$H_3$};
\draw (1.34,0.82) node[anchor=north west] {$H_4$};
\draw (3.45,0.77) node[anchor=north west] {$H_5$};
\draw (5.48,0.91) node[anchor=north west] {$H_6$};
\draw (-5.45,-2.13) node[anchor=north west] {$H_7$};
\draw (-3.2,-2.11) node[anchor=north west] {$H_8$};
\draw (-0.79,-2.2) node[anchor=north west] {$H_9$};
\draw (1.43,-2.17) node[anchor=north west] {$H_{10}$};
\draw (3.5,-2.17) node[anchor=north west] {$H_{11}$};
\draw (5.6,-2.16) node[anchor=north west] {$H_{12}$};
\draw (-3.97,-0.64)-- (-3.97,-1.49);
\draw (-1.77,-0.63)-- (-1.77,-1.48);
\draw (0.55,-0.63)-- (0.55,-1.48);
\draw (2.7,-0.64)-- (2.7,-1.49);
\draw (4.83,-0.64)-- (4.83,-1.49);
\draw (6.87,-0.61)-- (6.87,-1.47);
\begin{scriptsize}
\fill [color=black] (-6,1) circle (1.5pt);
\draw[color=black] (-6.2,1.16) node {$s_2$};
\fill [color=black] (-5,1) circle (1.5pt);
\draw[color=black] (-4.72,0.96) node {$s_3$};
\fill [color=black] (-5.45,2) circle (1.5pt);
\draw[color=black] (-5.36,2.16) node {$s_1$};
\fill [color=black] (-5.45,2.98) circle (1.5pt);
\draw[color=black] (-5.36,3.15) node {$v$};
\fill [color=black] (-4,2.37) circle (1.5pt);
\draw[color=black] (-3.91,2.54) node {$a_1$};
\fill [color=black] (-4,1.52) circle (1.5pt);
\draw[color=black] (-3.92,1.68) node {$a_2$};
\fill [color=black] (-3.8,1.01) circle (1.5pt);
\draw[color=black] (-4.0,1.17) node {$s_2$};
\fill [color=black] (-2.8,1.01) circle (1.5pt);
\draw[color=black] (-2.53,0.96) node {$s_3$};
\fill [color=black] (-3.25,2.01) circle (1.5pt);
\draw[color=black] (-3.09,2.18) node {$s_1$};
\fill [color=black] (-3.25,2.99) circle (1.5pt);
\draw[color=black] (-3.09,3.16) node {$v$};
\fill [color=black] (-1.8,2.38) circle (1.5pt);
\draw[color=black] (-1.64,2.55) node {$a_1$};
\fill [color=black] (-1.8,1.53) circle (1.5pt);
\draw[color=black] (-1.65,1.7) node {$a_2$};
\fill [color=black] (-1.58,1.02) circle (1.5pt);
\draw[color=black] (-1.7,1.18) node {$s_2$};
\fill [color=black] (-0.49,1.01) circle (1.5pt);
\draw[color=black] (-0.22,0.97) node {$s_3$};
\fill [color=black] (-0.94,2.01) circle (1.5pt);
\draw[color=black] (-0.78,2.18) node {$s_1$};
\fill [color=black] (-0.94,2.99) circle (1.5pt);
\draw[color=black] (-0.78,3.16) node {$v$};
\fill [color=black] (0.52,2.38) circle (1.5pt);
\draw[color=black] (0.68,2.55) node {$a_1$};
\fill [color=black] (0.52,1.53) circle (1.5pt);
\draw[color=black] (0.66,1.7) node {$a_2$};
\fill [color=black] (0.66,1) circle (1.5pt);
\draw[color=black] (0.5,1.16) node {$s_2$};
\fill [color=black] (1.66,1) circle (1.5pt);
\draw[color=black] (1.93,0.96) node {$s_3$};
\fill [color=black] (1.21,2) circle (1.5pt);
\draw[color=black] (1.37,2.16) node {$s_1$};
\fill [color=black] (1.21,2.98) circle (1.5pt);
\draw[color=black] (1.37,3.15) node {$v$};
\fill [color=black] (2.67,2.37) circle (1.5pt);
\draw[color=black] (2.83,2.54) node {$a_1$};
\fill [color=black] (2.67,1.52) circle (1.5pt);
\draw[color=black] (2.82,1.68) node {$a_2$};
\fill [color=black] (2.79,1) circle (1.5pt);
\draw[color=black] (2.65,1.16) node {$s_2$};
\fill [color=black] (3.79,1) circle (1.5pt);
\draw[color=black] (4.06,0.96) node {$s_3$};
\fill [color=black] (3.34,2) circle (1.5pt);
\draw[color=black] (3.5,2.16) node {$s_1$};
\fill [color=black] (3.34,2.98) circle (1.5pt);
\draw[color=black] (3.5,3.15) node {$v$};
\fill [color=black] (4.8,2.37) circle (1.5pt);
\draw[color=black] (4.96,2.54) node {$a_1$};
\fill [color=black] (4.8,1.52) circle (1.5pt);
\draw[color=black] (4.75,1.68) node {$a_2$};
\fill [color=black] (4.84,1.03) circle (1.5pt);
\draw[color=black] (4.7,1.19) node {$s_2$};
\fill [color=black] (5.84,1.03) circle (1.5pt);
\draw[color=black] (6.1,0.97) node {$s_3$};
\fill [color=black] (5.39,2.03) circle (1.5pt);
\draw[color=black] (5.55,2.19) node {$s_1$};
\fill [color=black] (5.39,3.01) circle (1.5pt);
\draw[color=black] (5.55,3.17) node {$v$};
\fill [color=black] (6.84,2.4) circle (1.5pt);
\draw[color=black] (7,2.56) node {$a_1$};
\fill [color=black] (6.84,1.55) circle (1.5pt);
\draw[color=black] (6.99,1.71) node {$a_2$};
\fill [color=black] (-5.97,-2.01) circle (1.5pt);
\draw[color=black] (-6.1,-1.84) node {$s_2$};
\fill [color=black] (-4.97,-2.01) circle (1.5pt);
\draw[color=black] (-4.71,-2.04) node {$s_3$};
\fill [color=black] (-5.42,-1.01) circle (1.5pt);
\draw[color=black] (-5.25,-0.85) node {$s_1$};
\fill [color=black] (-5.42,-0.03) circle (1.5pt);
\draw[color=black] (-5.25,0.13) node {$v$};
\fill [color=black] (-3.97,-0.64) circle (1.5pt);
\draw[color=black] (-3.8,-0.48) node {$a_1$};
\fill [color=black] (-3.97,-1.49) circle (1.5pt);
\draw[color=black] (-3.81,-1.33) node {$a_2$};
\fill [color=black] (-3.77,-2) circle (1.5pt);
\draw[color=black] (-3.9,-1.84) node {$s_2$};
\fill [color=black] (-2.77,-2) circle (1.5pt);
\draw[color=black] (-2.5,-2.04) node {$s_3$};
\fill [color=black] (-3.22,-1) circle (1.5pt);
\draw[color=black] (-3.06,-0.84) node {$s_1$};
\fill [color=black] (-3.22,-0.02) circle (1.5pt);
\draw[color=black] (-3.06,0.14) node {$v$};
\fill [color=black] (-1.77,-0.63) circle (1.5pt);
\draw[color=black] (-1.61,-0.47) node {$a_1$};
\fill [color=black] (-1.77,-1.48) circle (1.5pt);
\draw[color=black] (-1.62,-1.32) node {$a_2$};
\fill [color=black] (-1.55,-1.99) circle (1.5pt);
\draw[color=black] (-1.7,-1.83) node {$s_2$};
\fill [color=black] (-0.45,-2) circle (1.5pt);
\draw[color=black] (-0.19,-2.04) node {$s_3$};
\fill [color=black] (-0.9,-1) circle (1.5pt);
\draw[color=black] (-0.74,-0.84) node {$s_1$};
\fill [color=black] (-0.9,-0.02) circle (1.5pt);
\draw[color=black] (-0.74,0.14) node {$v$};
\fill [color=black] (0.55,-0.63) circle (1.5pt);
\draw[color=black] (0.71,-0.47) node {$a_1$};
\fill [color=black] (0.55,-1.48) circle (1.5pt);
\draw[color=black] (0.7,-1.32) node {$a_2$};
\fill [color=black] (0.7,-2.01) circle (1.5pt);
\draw[color=black] (0.55,-1.85) node {$s_2$};
\fill [color=black] (1.7,-2.01) circle (1.5pt);
\draw[color=black] (1.97,-2.05) node {$s_3$};
\fill [color=black] (1.25,-1.01) circle (1.5pt);
\draw[color=black] (1.41,-0.85) node {$s_1$};
\fill [color=black] (1.25,-0.03) circle (1.5pt);
\draw[color=black] (1.41,0.13) node {$v$};
\fill [color=black] (2.7,-0.64) circle (1.5pt);
\draw[color=black] (2.86,-0.48) node {$a_1$};
\fill [color=black] (2.7,-1.49) circle (1.5pt);
\draw[color=black] (2.85,-1.33) node {$a_2$};
\fill [color=black] (2.83,-2.01) circle (1.5pt);
\draw[color=black] (2.7,-1.85) node {$s_2$};
\fill [color=black] (3.83,-2.01) circle (1.5pt);
\draw[color=black] (4.17,-2.05) node {$s_3$};
\fill [color=black] (3.38,-1.01) circle (1.5pt);
\draw[color=black] (3.62,-0.85) node {$s_1$};
\fill [color=black] (3.38,-0.03) circle (1.5pt);
\draw[color=black] (3.62,0.13) node {$v$};
\fill [color=black] (4.83,-0.64) circle (1.5pt);
\draw[color=black] (5.07,-0.48) node {$a_1$};
\fill [color=black] (4.83,-1.49) circle (1.5pt);
\draw[color=black] (4.6,-1.33) node {$a_2$};
\fill [color=black] (4.87,-1.99) circle (1.5pt);
\draw[color=black] (4.7,-1.82) node {$s_2$};
\fill [color=black] (5.87,-1.99) circle (1.5pt);
\draw[color=black] (6.21,-2.04) node {$s_3$};
\fill [color=black] (5.42,-0.98) circle (1.5pt);
\draw[color=black] (5.67,-0.82) node {$s_1$};
\fill [color=black] (5.42,0) circle (1.5pt);
\draw[color=black] (5.67,0.17) node {$v$};
\fill [color=black] (6.87,-0.61) circle (1.5pt);
\draw[color=black] (7.11,-0.45) node {$a_1$};
\fill [color=black] (6.87,-1.47) circle (1.5pt);
\draw[color=black] (7.11,-1.3) node {$a_2$};
\end{scriptsize}
\end{tikzpicture}
\caption{Graphs $H_1, \ldots, H_{12}$.}
\label{nfig1}
\end{center}
\end{figure}

If $G[S\cup A\cup\{v\}]$ is either $H_1$ or $H_7$, then since $N_G(v)\cap\,(Q-S)=Q-S$ and $\omega(G)=k$, there exists a vertex $u\in Q-S$ such that  $uv\in E(G)$ and $ua_1\not\in E(G)$. Thus, $V(G) - \{s_1, s_2, a_1, v\}$ is a local resolving set for $G$. If it is either $H_2$, $H_3$, $H_4$, $H_5$, $H_8$, $H_9$, or $H_{11}$, then $V(G) - \{s_1, s_2, a_2, v\}$ is a local resolving set for $G$. And if it is either $H_6$ or $H_{12}$, then $V(G) - \{s_1, s_2, a_1, v\}$ is a local resolving set for $G$. 

If $|N_G(v) \cap A| = 1$ and $\{|N_G(s_h)\cap A|:\ h\in[3]\}=\{1,2\}$, then by applying  the assumptions on both $S$ and $A$, one can see that $G[S\cup A\cup\{v\}]$ is one of the graphs illustrated in Fig.~\ref{nfig2}.  

\begin{figure}[ht!]
\begin{center}
\begin{tikzpicture}
\clip(-6.5,0) rectangle (7,3.3);
\draw (-6,1)-- (-5,1);
\draw (-5,1)-- (-5.45,2);
\draw (-5.45,2)-- (-6,1);
\draw (-5.45,2.98)-- (-6,1);
\draw (-5.45,2.98)-- (-5.45,2);
\draw (-5.45,2.98)-- (-4,2.37);
\draw (-5.45,2)-- (-4,2.37);
\draw (-5.45,2)-- (-4,1.52);
\draw (-6,1)-- (-4,2.37);
\draw (-3.99,1.01)-- (-2.99,1.01);
\draw (-2.99,1.01)-- (-3.44,2.01);
\draw (-3.44,2.01)-- (-3.99,1.01);
\draw (-3.44,2.99)-- (-3.99,1.01);
\draw (-3.44,2.99)-- (-3.44,2.01);
\draw (-3.44,2.99)-- (-1.99,2.38);
\draw (-3.44,2.01)-- (-1.99,2.38);
\draw (-3.44,2.01)-- (-1.99,1.53);
\draw (-3.99,1.01)-- (-1.99,1.53);
\draw (-2,1)-- (-0.81,1);
\draw (-0.81,1)-- (-1.26,2);
\draw (-1.26,2)-- (-2,1);
\draw (-1.26,2.98)-- (-2,1);
\draw (-1.26,2.98)-- (-1.26,2);
\draw (-1.26,2.98)-- (0.19,2.37);
\draw (-0.81,1)-- (0.19,1.52);
\draw (-0.81,1)-- (0.19,2.37);
\draw (-5.66,0.9) node[anchor=north west] {$F_1$};
\draw (-3.39,0.85) node[anchor=north west] {$F_2$};
\draw (-1.09,0.84) node[anchor=north west] {$F_3$};
\draw (-5,1)-- (-4,1.52);
\draw (-2.99,1.01)-- (-1.99,2.38);
\draw (-1.26,2)-- (0.19,2.37);
\draw (-2,1)-- (0.19,1.52);
\draw (0.21,1)-- (1.21,1);
\draw (1.21,1)-- (0.76,2);
\draw (0.76,2)-- (0.21,1);
\draw (0.76,2.98)-- (0.21,1);
\draw (0.76,2.98)-- (0.76,2);
\draw (0.76,2.98)-- (2.21,2.37);
\draw (0.76,2)-- (2.21,2.37);
\draw (0.76,2)-- (2.21,1.52);
\draw (0.21,1)-- (2.21,2.37);
\draw (2.22,1.01)-- (3.22,1.01);
\draw (3.22,1.01)-- (2.77,2.01);
\draw (2.77,2.01)-- (2.22,1.01);
\draw (2.77,2.99)-- (2.22,1.01);
\draw (2.77,2.99)-- (2.77,2.01);
\draw (2.77,2.99)-- (4.22,2.38);
\draw (2.77,2.01)-- (4.22,2.38);
\draw (2.77,2.01)-- (4.22,1.53);
\draw (2.22,1.01)-- (4.22,1.53);
\draw (4.2,1)-- (5.39,0.99);
\draw (5.39,0.99)-- (4.94,1.99);
\draw (4.94,1.99)-- (4.2,1);
\draw (4.94,2.98)-- (4.2,1);
\draw (4.94,2.98)-- (4.94,1.99);
\draw (4.94,2.98)-- (6.39,2.36);
\draw (5.39,0.99)-- (6.39,1.51);
\draw (5.39,0.99)-- (6.39,2.36);
\draw (0.54,0.9) node[anchor=north west] {$F_4$};
\draw (2.82,0.85) node[anchor=north west] {$F_5$};
\draw (5.11,0.84) node[anchor=north west] {$F_6$};
\draw (1.21,1)-- (2.21,1.52);
\draw (3.22,1.01)-- (4.22,2.38);
\draw (4.94,1.99)-- (6.39,2.36);
\draw (4.2,1)-- (6.39,1.51);
\draw (2.21,2.37)-- (2.21,1.52);
\draw (4.22,2.38)-- (4.22,1.53);
\draw (6.39,2.36)-- (6.39,1.51);
\begin{scriptsize}
\fill [color=black] (-6,1) circle (1.5pt);
\draw[color=black] (-6.2,1.11) node {$s_2$};
\fill [color=black] (-5,1) circle (1.5pt);
\draw[color=black] (-4.92,1.16) node {$s_3$};
\fill [color=black] (-5.45,2) circle (1.5pt);
\draw[color=black] (-5.25,2.16) node {$s_1$};
\fill [color=black] (-5.45,2.98) circle (1.5pt);
\draw[color=black] (-5.36,3.15) node {$v$};
\fill [color=black] (-4,2.37) circle (1.5pt);
\draw[color=black] (-3.91,2.54) node {$a_1$};
\fill [color=black] (-4,1.52) circle (1.5pt);
\draw[color=black] (-4,1.68) node {$a_2$};
\fill [color=black] (-3.99,1.01) circle (1.5pt);
\draw[color=black] (-4.15,1.15) node {$s_2$};
\fill [color=black] (-2.99,1.01) circle (1.5pt);
\draw[color=black] (-2.76,1.06) node {$s_3$};
\fill [color=black] (-3.44,2.01) circle (1.5pt);
\draw[color=black] (-3.29,2.24) node {$s_1$};
\fill [color=black] (-3.44,2.99) circle (1.5pt);
\draw[color=black] (-3.28,3.16) node {$v$};
\fill [color=black] (-1.99,2.38) circle (1.5pt);
\draw[color=black] (-1.83,2.55) node {$a_1$};
\fill [color=black] (-1.99,1.53) circle (1.5pt);
\draw[color=black] (-1.95,1.7) node {$a_2$};
\fill [color=black] (-2,1) circle (1.5pt);
\draw[color=black] (-2.2,1.11) node {$s_2$};
\fill [color=black] (-0.81,1) circle (1.5pt);
\draw[color=black] (-0.57,0.95) node {$s_3$};
\fill [color=black] (-1.26,2) circle (1.5pt);
\draw[color=black] (-1.12,2.2) node {$s_1$};
\fill [color=black] (-1.26,2.98) circle (1.5pt);
\draw[color=black] (-1.1,3.15) node {$v$};
\fill [color=black] (0.19,2.37) circle (1.5pt);
\draw[color=black] (0.27,2.59) node {$a_1$};
\fill [color=black] (0.19,1.52) circle (1.5pt);
\draw[color=black] (0.23,1.73) node {$a_2$};
\fill [color=black] (0.21,1) circle (1.5pt);
\draw[color=black] (0.0,1.16) node {$s_2$};
\fill [color=black] (1.21,1) circle (1.5pt);
\draw[color=black] (1.31,1.23) node {$s_3$};
\fill [color=black] (0.76,2) circle (1.5pt);
\draw[color=black] (0.92,2.24) node {$s_1$};
\fill [color=black] (0.76,2.98) circle (1.5pt);
\draw[color=black] (0.92,3.15) node {$v$};
\fill [color=black] (2.21,2.37) circle (1.5pt);
\draw[color=black] (2.37,2.54) node {$a_1$};
\fill [color=black] (2.21,1.52) circle (1.5pt);
\draw[color=black] (2.05,1.74) node {$a_2$};
\fill [color=black] (2.22,1.01) circle (1.5pt);
\draw[color=black] (2.0,1.15) node {$s_2$};
\fill [color=black] (3.22,1.01) circle (1.5pt);
\draw[color=black] (3.53,1.11) node {$s_3$};
\fill [color=black] (2.77,2.01) circle (1.5pt);
\draw[color=black] (2.92,2.25) node {$s_1$};
\fill [color=black] (2.77,2.99) circle (1.5pt);
\draw[color=black] (2.93,3.16) node {$v$};
\fill [color=black] (4.22,2.38) circle (1.5pt);
\draw[color=black] (4.26,2.62) node {$a_1$};
\fill [color=black] (4.22,1.53) circle (1.5pt);
\draw[color=black] (4.08,1.77) node {$a_2$};
\fill [color=black] (4.2,1) circle (1.5pt);
\draw[color=black] (4.05,1.17) node {$s_2$};
\fill [color=black] (5.39,0.99) circle (1.5pt);
\draw[color=black] (5.68,1.0) node {$s_3$};
\fill [color=black] (4.94,1.99) circle (1.5pt);
\draw[color=black] (5.1,2.24) node {$s_1$};
\fill [color=black] (4.94,2.98) circle (1.5pt);
\draw[color=black] (5.11,3.14) node {$v$};
\fill [color=black] (6.39,2.36) circle (1.5pt);
\draw[color=black] (6.56,2.53) node {$a_1$};
\fill [color=black] (6.39,1.51) circle (1.5pt);
\draw[color=black] (6.55,1.67) node {$a_2$};
\end{scriptsize}
\end{tikzpicture}
\caption{Graphs $F_1,\ldots, F_6$.}
\label{nfig2}
\end{center}
\end{figure}
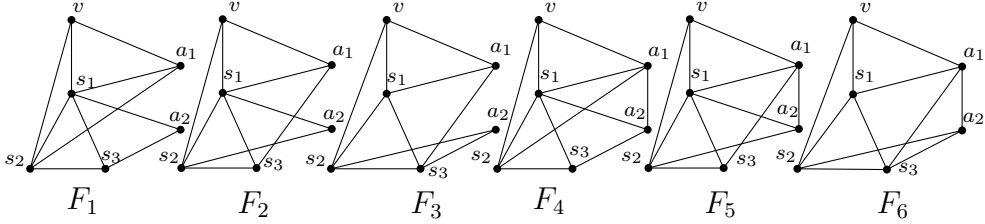

If $G[S\cup A\cup\{v\}]$ is either $F_1$ or $F_4$, then since $N_G(v)\cap\,(Q-S)=Q-S$ and $\omega(G)=k$, there is a  vertex such as $u$ in $Q-S$ that  $uv\in E(G)$ and $ua_1\not\in E(G)$. Thus, $V(G) - \{s_1, s_2, a_1, v\}$ is a local resolving set for $G$. If it is either $F_2$ or $F_5$, then $V(G) - \{s_1, s_2, a_2, v\}$ is a local resolving set for $G$. If it is $F_3$, then $V(G) - \{s_1, s_2, s_3, a_2\}$ is a local resolving set for $G$. If it is $F_6$, then $V(G) - \{s_1, s_2, a_1,v\}$ is a local resolving set for $G$.

If  $|N_G(v) \cap A| = 2$ and for each $i\in[3]$, $|N_G(s_i)\cap\,A|\leq1$, then $V(G) - \{s_1, s_2, s_3,v\}$ is a local resolving set for $G$.

If $|N_G(v) \cap A| = 2$ and $\{|N_G(s_h)\cap A|:h\in[3]\}=\{0,1,2\}$, then by employing the assumptions on both $S$ and $A$, one can observe  that $G[S\cup A\cup\{v\}]$ is one of the graphs illustrated in Fig.~\ref{nfig3}. 

\begin{figure}[ht!]
\begin{center}
\begin{tikzpicture}
\clip(-6.5,0) rectangle (7,3.5);
\draw (-6,1)-- (-5,1);
\draw (-5,1)-- (-5.45,2);
\draw (-5.45,2)-- (-6,1);
\draw (-5.45,2.98)-- (-6,1);
\draw (-5.45,2.98)-- (-5.45,2);
\draw (-5.45,2.98)-- (-4,2.37);
\draw (-5.45,2)-- (-4,2.37);
\draw (-5.45,2)-- (-4,1.52);
\draw (-6,1)-- (-4,2.37);
\draw (-3.99,1.01)-- (-2.99,1.01);
\draw (-2.99,1.01)-- (-3.44,2.01);
\draw (-3.44,2.01)-- (-3.99,1.01);
\draw (-3.44,2.99)-- (-3.99,1.01);
\draw (-3.44,2.99)-- (-3.44,2.01);
\draw (-3.44,2.99)-- (-1.99,2.38);
\draw (-3.44,2.01)-- (-1.99,2.38);
\draw (-3.44,2.01)-- (-1.99,1.53);
\draw (-2,1)-- (-0.81,1);
\draw (-0.81,1)-- (-1.26,2);
\draw (-1.26,2)-- (-2,1);
\draw (-1.26,2.98)-- (-2,1);
\draw (-1.26,2.98)-- (-1.26,2);
\draw (-1.26,2.98)-- (0.19,2.37);
\draw (-0.81,1)-- (0.19,1.52);
\draw (-0.81,1)-- (0.19,2.37);
\draw (-5.66,0.9) node[anchor=north west] {$R_1$};
\draw (-3.39,0.85) node[anchor=north west] {$R_2$};
\draw (-1.09,0.84) node[anchor=north west] {$R_3$};
\draw (-2.99,1.01)-- (-1.99,2.38);
\draw (-1.26,2)-- (0.19,2.37);
\draw (0.54,0.9) node[anchor=north west] {$R_4$};
\draw (2.82,0.85) node[anchor=north west] {$R_5$};
\draw (5.11,0.84) node[anchor=north west] {$R_6$};
\draw (-5.45,2.98)-- (-4,1.52);
\draw (-3.44,2.99)-- (-1.99,1.53);
\draw (-1.26,2.98)-- (0.19,1.52);
\draw (0.25,1.02)-- (1.25,1.02);
\draw (1.25,1.02)-- (0.8,2.02);
\draw (0.8,2.02)-- (0.25,1.02);
\draw (0.8,3)-- (0.25,1.02);
\draw (0.8,3)-- (0.8,2.02);
\draw (0.8,3)-- (2.25,2.39);
\draw (0.8,2.02)-- (2.25,2.39);
\draw (0.8,2.02)-- (2.25,1.54);
\draw (0.25,1.02)-- (2.25,2.39);
\draw (2.26,1.03)-- (3.26,1.03);
\draw (3.26,1.03)-- (2.81,2.03);
\draw (2.81,2.03)-- (2.26,1.03);
\draw (2.81,3.01)-- (2.26,1.03);
\draw (2.81,3.01)-- (2.81,2.03);
\draw (2.81,3.01)-- (4.26,2.4);
\draw (2.81,2.03)-- (4.26,2.4);
\draw (2.81,2.03)-- (4.26,1.55);
\draw (4.25,1.02)-- (5.44,1.01);
\draw (5.44,1.01)-- (4.98,2.02);
\draw (4.98,2.02)-- (4.25,1.02);
\draw (4.98,3)-- (4.25,1.02);
\draw (4.98,3)-- (4.98,2.02);
\draw (4.98,3)-- (6.44,2.39);
\draw (5.44,1.01)-- (6.44,1.53);
\draw (5.44,1.01)-- (6.44,2.39);
\draw (3.26,1.03)-- (4.26,2.4);
\draw (4.98,2.02)-- (6.44,2.39);
\draw (0.8,3)-- (2.25,1.54);
\draw (2.81,3.01)-- (4.26,1.55);
\draw (4.98,3)-- (6.44,1.53);
\draw (2.25,2.39)-- (2.25,1.54);
\draw (4.26,2.4)-- (4.26,1.55);
\draw (6.44,2.39)-- (6.44,1.53);
\begin{scriptsize}
\fill [color=black] (-6,1) circle (1.5pt);
\draw[color=black] (-6.2,1.11) node {$s_2$};
\fill [color=black] (-5,1) circle (1.5pt);
\draw[color=black] (-4.92,1.16) node {$s_3$};
\fill [color=black] (-5.45,2) circle (1.5pt);
\draw[color=black] (-5.26,2.16) node {$s_1$};
\fill [color=black] (-5.45,2.98) circle (1.5pt);
\draw[color=black] (-5.36,3.15) node {$v$};
\fill [color=black] (-4,2.37) circle (1.5pt);
\draw[color=black] (-3.91,2.54) node {$a_1$};
\fill [color=black] (-4,1.52) circle (1.5pt);
\draw[color=black] (-4.05,1.72) node {$a_2$};
\fill [color=black] (-3.99,1.01) circle (1.5pt);
\draw[color=black] (-4.2,1.15) node {$s_2$};
\fill [color=black] (-2.99,1.01) circle (1.5pt);
\draw[color=black] (-2.76,1.06) node {$s_3$};
\fill [color=black] (-3.44,2.01) circle (1.5pt);
\draw[color=black] (-3.29,2.24) node {$s_1$};
\fill [color=black] (-3.44,2.99) circle (1.5pt);
\draw[color=black] (-3.28,3.16) node {$v$};
\fill [color=black] (-1.99,2.38) circle (1.5pt);
\draw[color=black] (-1.83,2.55) node {$a_1$};
\fill [color=black] (-1.99,1.53) circle (1.5pt);
\draw[color=black] (-2,1.7) node {$a_2$};
\fill [color=black] (-2,1) circle (1.5pt);
\draw[color=black] (-2.2,1.11) node {$s_2$};
\fill [color=black] (-0.81,1) circle (1.5pt);
\draw[color=black] (-0.57,0.9) node {$s_3$};
\fill [color=black] (-1.26,2) circle (1.5pt);
\draw[color=black] (-1.12,2.2) node {$s_1$};
\fill [color=black] (-1.26,2.98) circle (1.5pt);
\draw[color=black] (-1.1,3.15) node {$v$};
\fill [color=black] (0.19,2.37) circle (1.5pt);
\draw[color=black] (0.27,2.59) node {$a_1$};
\fill [color=black] (0.19,1.52) circle (1.5pt);
\draw[color=black] (0.23,1.73) node {$a_2$};
\fill [color=black] (0.25,1.02) circle (1.5pt);
\draw[color=black] (0.1,1.13) node {$s_2$};
\fill [color=black] (1.25,1.02) circle (1.5pt);
\draw[color=black] (1.42,1.18) node {$s_3$};
\fill [color=black] (0.8,2.02) circle (1.5pt);
\draw[color=black] (0.96,2.19) node {$s_1$};
\fill [color=black] (0.8,3) circle (1.5pt);
\draw[color=black] (0.96,3.17) node {$v$};
\fill [color=black] (2.25,2.39) circle (1.5pt);
\draw[color=black] (2.41,2.56) node {$a_1$};
\fill [color=black] (2.25,1.54) circle (1.5pt);
\draw[color=black] (2.15,1.8) node {$a_2$};
\fill [color=black] (2.26,1.03) circle (1.5pt);
\draw[color=black] (2.05,1.17) node {$s_2$};
\fill [color=black] (3.26,1.03) circle (1.5pt);
\draw[color=black] (3.49,1.08) node {$s_3$};
\fill [color=black] (2.81,2.03) circle (1.5pt);
\draw[color=black] (2.96,2.26) node {$s_1$};
\fill [color=black] (2.81,3.01) circle (1.5pt);
\draw[color=black] (2.97,3.18) node {$v$};
\fill [color=black] (4.26,2.4) circle (1.5pt);
\draw[color=black] (4.42,2.57) node {$a_1$};
\fill [color=black] (4.26,1.55) circle (1.5pt);
\draw[color=black] (4.1,1.8) node {$a_2$};
\fill [color=black] (4.25,1.02) circle (1.5pt);
\draw[color=black] (4.0,1.13) node {$s_2$};
\fill [color=black] (5.44,1.01) circle (1.5pt);
\draw[color=black] (5.68,0.95) node {$s_3$};
\fill [color=black] (4.98,2.02) circle (1.5pt);
\draw[color=black] (5.14,2.21) node {$s_1$};
\fill [color=black] (4.98,3) circle (1.5pt);
\draw[color=black] (5.15,3.16) node {$v$};
\fill [color=black] (6.44,2.39) circle (1.5pt);
\draw[color=black] (6.53,2.6) node {$a_1$};
\fill [color=black] (6.44,1.53) circle (1.5pt);
\draw[color=black] (6.65,1.7) node {$a_2$};
\end{scriptsize}
\end{tikzpicture}
\caption{Graphs $R_1,\ldots, R_6$.}
\label{nfig3}
\end{center}
\end{figure}
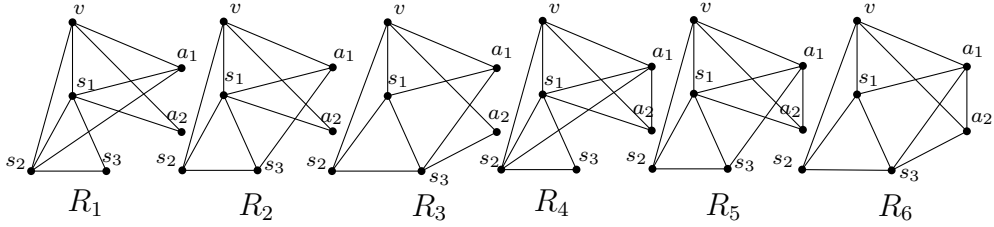

If $G[S\cup A\cup\{v\}]$ is either $R_1$ or $R_4$, then since $N_G(v)\cap\,(Q-S)=Q-S$ and $\omega(G)=k$, there is a  vertex such as $u$ in $Q-S$ that  $uv\in E(G)$ and $ua_1\not\in E(G)$. Thus, $V(G) - \{s_1, s_2, a_1, v\}$ is a local resolving set for $G$. If it is either $R_2$, $R_3$, or $R_6$, then $V(G) - \{s_1, s_2, a_2, v\}$ is a local resolving set for $G$. If it is $R_5$, then since $N_G(v)\cap\,(Q-S)=Q-S$ and $\omega(G)=k$, there is a  vertex such as $u$ in $Q-S$ that either $ua_1\not\in E(G)$ or $ua_2\not\in E(G)$ holds. So, if $ua_1\not\in E(G)$, then $V(G) - \{s_1, s_2, a_1, v\}$ is a local resolving set for $G$. If $ua_2\not\in E(G)$, then $V(G) - \{s_1, s_2, a_2, v\}$ is a local resolving set for $G$.

If $|N_G(v) \cap A| = 2$ and $\{|N_G(s_h)\cap A|:h\in[3]\}=\{1,2\}$, then by applying  the assumptions on both $S$ and $A$, one can see that $G[S\cup A\cup\{v\}]$ is one of the graphs illustrated in Fig.~\ref{nfig4}. 

\begin{figure}[ht!]
\begin{center}
\begin{tikzpicture}
\clip(-6.5,0) rectangle (4,3.5);
\draw (-6,1)-- (-5,1);
\draw (-5,1)-- (-5.45,2);
\draw (-5.45,2)-- (-6,1);
\draw (-5.45,2.98)-- (-6,1);
\draw (-5.45,2.98)-- (-5.45,2);
\draw (-5.45,2.98)-- (-4,2.37);
\draw (-5.45,2)-- (-4,2.37);
\draw (-5.45,2)-- (-4,1.52);
\draw (-6,1)-- (-4,2.37);
\draw (-3.71,1)-- (-2.52,1);
\draw (-2.52,1)-- (-2.97,2);
\draw (-2.97,2)-- (-3.71,1);
\draw (-2.97,2.98)-- (-3.71,1);
\draw (-2.97,2.98)-- (-2.97,2);
\draw (-2.97,2.98)-- (-1.52,2.37);
\draw (-2.52,1)-- (-1.52,1.52);
\draw (-2.52,1)-- (-1.52,2.37);
\draw (-5.66,0.9) node[anchor=north west] {$J_1$};
\draw (-3.12,0.88) node[anchor=north west] {$J_2$};
\draw (-0.65,0.84) node[anchor=north west] {$J_3$};
\draw (-2.97,2)-- (-1.52,2.37);
\draw (1.77,0.83) node[anchor=north west] {$J_4$};
\draw (-5.45,2.98)-- (-4,1.52);
\draw (-2.97,2.98)-- (-1.52,1.52);
\draw (-5,1)-- (-4,1.52);
\draw (-3.71,1)-- (-1.52,1.52);
\draw (-1.2,1.02)-- (-0.2,1.02);
\draw (-0.2,1.02)-- (-0.65,2.02);
\draw (-0.65,2.02)-- (-1.2,1.02);
\draw (-0.65,3)-- (-1.2,1.02);
\draw (-0.65,3)-- (-0.65,2.02);
\draw (-0.65,3)-- (0.8,2.39);
\draw (-0.65,2.02)-- (0.8,2.39);
\draw (-0.65,2.02)-- (0.8,1.54);
\draw (-1.2,1.02)-- (0.8,2.39);
\draw (1.09,1.02)-- (2.28,1.01);
\draw (2.28,1.01)-- (1.83,2.02);
\draw (1.83,2.02)-- (1.09,1.02);
\draw (1.83,3)-- (1.09,1.02);
\draw (1.83,3)-- (1.83,2.02);
\draw (1.83,3)-- (3.28,2.39);
\draw (2.28,1.01)-- (3.28,1.53);
\draw (2.28,1.01)-- (3.28,2.39);
\draw (1.83,2.02)-- (3.28,2.39);
\draw (-0.65,3)-- (0.8,1.54);
\draw (1.83,3)-- (3.28,1.53);
\draw (-0.2,1.02)-- (0.8,1.54);
\draw (1.09,1.02)-- (3.28,1.53);
\draw (0.8,2.39)-- (0.8,1.54);
\draw (3.28,2.39)-- (3.28,1.53);
\begin{scriptsize}
\fill [color=black] (-6,1) circle (1.5pt);
\draw[color=black] (-6.2,1.11) node {$s_2$};
\fill [color=black] (-5,1) circle (1.5pt);
\draw[color=black] (-4.92,1.16) node {$s_3$};
\fill [color=black] (-5.45,2) circle (1.5pt);
\draw[color=black] (-5.27,2.16) node {$s_1$};
\fill [color=black] (-5.45,2.98) circle (1.5pt);
\draw[color=black] (-5.36,3.15) node {$v$};
\fill [color=black] (-4,2.37) circle (1.5pt);
\draw[color=black] (-3.91,2.54) node {$a_1$};
\fill [color=black] (-4,1.52) circle (1.5pt);
\draw[color=black] (-3.92,1.68) node {$a_2$};
\fill [color=black] (-3.71,1) circle (1.5pt);
\draw[color=black] (-3.9,1.11) node {$s_2$};
\fill [color=black] (-2.52,1) circle (1.5pt);
\draw[color=black] (-2.27,0.95) node {$s_3$};
\fill [color=black] (-2.97,2) circle (1.5pt);
\draw[color=black] (-2.82,2.2) node {$s_1$};
\fill [color=black] (-2.97,2.98) circle (1.5pt);
\draw[color=black] (-2.81,3.15) node {$v$};
\fill [color=black] (-1.52,2.37) circle (1.5pt);
\draw[color=black] (-1.43,2.59) node {$a_1$};
\fill [color=black] (-1.52,1.52) circle (1.5pt);
\draw[color=black] (-1.48,1.73) node {$a_2$};
\fill [color=black] (-1.2,1.02) circle (1.5pt);
\draw[color=black] (-1.35,1.13) node {$s_2$};
\fill [color=black] (-0.2,1.02) circle (1.5pt);
\draw[color=black] (-0.1,1.25) node {$s_3$};
\fill [color=black] (-0.65,2.02) circle (1.5pt);
\draw[color=black] (-0.49,2.19) node {$s_1$};
\fill [color=black] (-0.65,3) circle (1.5pt);
\draw[color=black] (-0.49,3.17) node {$v$};
\fill [color=black] (0.8,2.39) circle (1.5pt);
\draw[color=black] (0.96,2.56) node {$a_1$};
\fill [color=black] (0.8,1.54) circle (1.5pt);
\draw[color=black] (0.95,1.71) node {$a_2$};
\fill [color=black] (1.09,1.02) circle (1.5pt);
\draw[color=black] (0.9,1.13) node {$s_2$};
\fill [color=black] (2.28,1.01) circle (1.5pt);
\draw[color=black] (2.52,0.95) node {$s_3$};
\fill [color=black] (1.83,2.02) circle (1.5pt);
\draw[color=black] (1.99,2.21) node {$s_1$};
\fill [color=black] (1.83,3) circle (1.5pt);
\draw[color=black] (2,3.16) node {$v$};
\fill [color=black] (3.28,2.39) circle (1.5pt);
\draw[color=black] (3.37,2.6) node {$a_1$};
\fill [color=black] (3.28,1.53) circle (1.5pt);
\draw[color=black] (3.45,1.7) node {$a_2$};
\end{scriptsize}
\end{tikzpicture}
\caption{Graphs $J_1, \ldots, J_4$.}
\label{nfig4}
\end{center}
\end{figure}
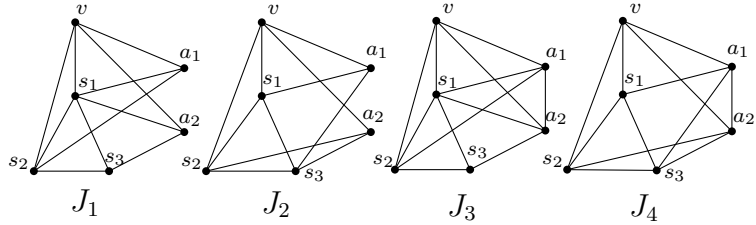

If $G[S\cup A\cup\{v\}]$ is either $J_1$ or $J_3$, then since $N_G(v)\cap\,(Q-S)=Q-S$ and $\omega(G)=k$, there is a  vertex such as $u$ in $Q-S$ that  $ua_1\not\in E(G)$. Thus, $V(G) - \{s_1, s_2, a_1, v\}$ is a local resolving set for $G$. If it is either $J_2$ or $J_4$, then $V(G) - \{s_1, s_2, s_3, v\}$ is a local resolving set for $G$.

\medskip\noindent
{\bf Case 3}: $|S| = 2$. \\
Let $S=\{s_1,s_2\}$ and assume without loss of generality that $|N_G(s_1)|\geq|N_G(s_2)|$. Since $\omega(G)=k$ and $|S|=2$, one can see that the following statements hold, otherwise we can increase the size of either $S$ or $Q$.

\begin{enumerate}
  \item[I.] $\big(N_G(s_1)\cap(V(G)-Q)\big)\cap\big(N_G(s_2)\cap(V(G)-Q)\big)=\emptyset$.
  \item[II.] For each $w\in Q-S$ and $i\in\{1,2\}$, $\big(N_G(s_i)\cap(V(G)-Q))\big)-\big(N_G(w)\cap(V(G)-Q))\big)=N_G(s_i)\cap(V(G)-Q))$ or $\emptyset$.
  \item[III.] If $N_G(s_1)\cap(V(G)-Q)\neq\emptyset$ and $N_G(s_2)\cap(V(G)-Q)\neq\emptyset$, then for each $w\in(Q-S)$, either $\big\{\big(N_G(s_1)\cap(V(G)-Q))\big)-\big(N_G(w)\cap(V(G)-Q))\big)=N_G(s_1)\cap(V(G)-Q))$ and $\big(N_G(s_2)\cap(V(G)-Q))\big)-\big(N_G(w)\cap(V(G)-Q))\big)=\emptyset\big\}$, or $\big\{\big(N_G(s_1)\cap(V(G)-Q))\big)-\big(N_G(w)\cap(V(G)-Q))\big)=\emptyset$ and $\big(N_G(s_2)\cap(V(G)-Q))\big)-\big(N_G(w)\cap(V(G)-Q))\big)=N_G(s_2)\cap(V(G)-Q))\big\}$.
  \item[IV.] For each $w\in(Q-S)$ and $x\in\big((V(G)-(N_G(s_1)\cup\,N_G(s_2))\big)$, $wx\not\in E(G)$.
\end{enumerate}

Now, we consider four sub-cases as follows. 

\medskip\noindent
{\bf Case 3.1}: For $i\in[2]$, there exist a subset $X$ of $N_G(s_i)\cap(V(G)-Q)$ such that $|X|=3$ and $|E(G[X])|\leq2$. \\
Suppose $X=\{x_1,x_2,x_3\}$. If $|E(G[X])|=0$, then by applying (I) we infer that the set $V(G)-(\{s_i\}\cup\,X)$ forms a local resolving set for $G$. If $E(G[X])=\{x_1x_2\}$ or $\{x_1x_2,x_2x_3\}$, then by employing both (II) and $\omega(G)=k$, we get that there is a vertex $w$ of $Q-S$ that is adjacent to no vertex of $X$. Thus, by applying (I), the set $V(G)-\{w,s_i,x_1,x_3\}$ is a  local resolving set for $G$.  The proof of other cases are similar and hence omitted. 

\medskip\noindent
{\bf Case 3.2}: $|N_G(s_1)\cap(V(G)-Q)|\geq2$ and $|N_G(s_2)\cap(V(G)-Q)|\geq1$. \\
Assume $a_1,a_2\in(N_G(s_1)\cap(V(G)-Q))$ and $b_1\in(N_G(s_2)\cap(V(G)-Q))$. If $a_1a_2\not\in E(G)$, then by using (I), the set $V(G)-\{s_1,a_1,a_2,b_1\}$ is a  local resolving set for $G$. If $a_1a_2\in E(G)$, then by applying (II), (IV), and $\omega(G)=k$,  there is a vertex $w$ of $Q-S$ that is adjacent to no vertex in the set $((V(G)-Q)-N_G(s_2))$. Now, if $wb_1\not\in E(G)$, then  $V(G)-\{w,s_1,a_1,b_1\}$  is a  local resolving set for $G$. If there exists $x\in((V(G)-Q)-N_G(s_2))$, such that  $xb_1\in 
 E(G)$, then by using (I), the set $V(G)-\{w,s_1,a_i,b_1\}$  is a  local resolving set for $G$, where $i\in [2]$ and $a_i\neq x$.  Otherwise, $N_G(b_1)\cap((V(G)-Q)-N_G(s_2))=\emptyset$ and by employing (II) and (III), there is no vertex $w'\in Q$ such that both $w'b_1$ and $\{w'a_1,w'a_2\}\subseteq  E(G)$ hold. In this case,  by using (I), both $V(G)-\{w,s_1,a_1,b_1\}$ and $V(G)-\{w,s_1,a_2,b_1\}$ are   local resolving sets for $G$.

\medskip\noindent
{\bf Case 3.3}:  $|N_G(s_1)\cap(V(G)-Q)|=|N_G(s_2)\cap(V(G)-Q)|=1$. \\
Assume that $N_G(s_1)\cap(V(G)-Q)=\{a_1\}$, $N_G(s_2)\cap(V(G)-Q)=\{b_1\}$, and let $A_1=N_G(a_1)\cap(V(G)-(Q\cup\{b_1\}))$, and $B_1=N_G(b_1)\cap(V(G)-(Q\cup\{a_1\}))$. If $Q-S\neq\emptyset$, $A_1\cup B_1 \neq \emptyset$, $w\in Q-S$, and $x\in A_1\cup B_1$, then by employing (III), either $(wa_1\in E(G)$ and $wb_1\not\in E(G))$ or $(wa_1\not\in E(G)$ and $wb_1\in E(G))$ holds. Without loss of generality, assume $wa_1\not\in E(G)$ and $wb_1\in E(G)$. In this case, by using (I) and (IV), we get that $V(G)-\{s_2,a_1,b_1,x\}$ is a local resolving set for $G$. If $A_1\cup\,B_1=\emptyset$, then by using (IV) and  the fact that  $G$ is connected, we have $\omega(G)=n-2$, but this is a contradiction with $\omega(G)\leq\,n-3$. If $Q-S=\emptyset$, then $\omega(G)=2$, and by employing Theorem~\ref{cth1} we have $\dim_l(G) \leq n-3$ and the equality holds if and only if $G \cong C_5$.

\medskip\noindent
{\bf Case 3.4}:  $|N_G(s_1)\cap(V(G)-Q)|\geq1$ and $|N_G(s_2)\cap(V(G)-Q)|=0$. \\
In this case, for $i\in[3]$, let $D_i$ be the set of vertices in $V(G)-Q$ at distance $i$ from $s_1$, that is, $D_i=\{v\in V(G)-Q:\ d_G(s_1,v)=i\}$. By applying  (IV), $\omega(G)\leq\,n(G)-3$, and $G$ is connected, hence  $|D_1|+|D_2|+|D_3|\geq 3$. If $D_i\neq\emptyset$ for $i\in[3]$, then by using (I) and (IV),  $V(G)-\{x_1,y_1,z_1,s_2\}$ is  local resolving set for $G$, where $x_1\in D_1$, $y_1\in D_2$, and $z_1\in D_3$. If $|D_1|\geq 2$,  $|D_2|\geq 1$, $x_1,x_2\in D_1$, $y_1\in D_2$, and $x_1x_2\not\in E(G)$, then by applying (I),  $V(G)-\{x_1,x_2,y_1,s_2\}$ is  local resolving set for $G$. If $|D_1|\geq2$,  $|D_2|\geq1$, $x_1,x_2\in D_1$, $y_1\in D_2$, and $x_1x_2\in E(G)$, then since $\omega(G)=k$, by using (II), there is a vertex $w$ of $Q-S$ that is adjacent to neither $x_1$ nor $x_2$. So, by applying (I) and (IV), both   $V(G)-\{x_1,y_1,s_1,w\}$ and $V(G)-\{x_2,y_1,s_1,w\}$ are  local resolving sets for $G$. If $|D_1|=1$,  $|D_2|\geq2$, $x_1\in D_1$, $y_1,y_2\in D_2$, and $y_1y_2\not\in E(G)$, then by applying (I), $V(G)-\{s_1,s_2,y_1,y_2\}$ is a  local resolving set for $G$. If $|D_1|=1$,  $|D_2|\geq2$, $x_1\in D_1$, $y_1,y_2\in D_2$, and $y_1y_2\in E(G)$, then $\omega(G)\geq3$ and $Q-S\neq\emptyset$. Thus, by applying (I) and (VI), $V(G)-\{s_1,s_2,x_1,y_1\}$ is a  local resolving set for $G$. If $|D_1|\geq3$,  $|D_2|=0$,  $x_1,x_2\in D_1$, and $x_1x_2\not\in E(G)$, then by employing Case 3.1,  $\dim_l(G)\leq\,n-4$. Otherwise,  $|D_1|\geq3$,  $|D_2|=0$,  and $G[D_1]\cong K_{|D_1|}$. In this case, by applying (II) and $\omega(G)\leq\,n(G)-3$, we infer that there are two positive integers $\lambda$ and $\mu$ such that $\lambda\geq\mu\geq2$ and $\lambda+\mu<n$, and $G\cong\,K_n^-(\lambda,\mu)$. Thus, Lemma \ref{lem:Kn-} gives the result.
\end{proof}

\begin{corollary}
If $G$ is a connected graph with $\omega(G) = n(G) - 3$, then 
$$n(G) - 8\le \dim_l(G) \le n(G) - 3\,.$$ 
Furthermore, both bounds are sharp.
\end{corollary}

\begin{proof}
The bounds follows by combining Theorem~\ref{th1} with Theorem~\ref{thm:Okamoto-2}. Theorem~\ref{th1} also yields sharpness of the upper bound. To complete the proof, we need to demonstrate that the lower bound is also sharp. Let $\Lambda$ be the graph illustrated in Figure \ref{nnfig0}. 

\begin{figure}[ht!]
\begin{center}
\begin{tikzpicture}
\clip(-7,0) rectangle (-3,5.5);
\draw (-6,1)-- (-5,1);
\draw (-5,1)-- (-4.29,1.71);
\draw (-4.29,1.71)-- (-4.29,2.71);
\draw (-4.29,2.71)-- (-5,3.41);
\draw (-5,3.41)-- (-6,3.41);
\draw (-6,3.41)-- (-6.71,2.71);
\draw (-6.71,2.71)-- (-6.71,1.71);
\draw (-6.71,1.71)-- (-6,1);
\draw (-6.71,1.71)-- (-5,5);
\draw (-6.71,1.71)-- (-6,5);
\draw (-6.71,1.71)-- (-4,5);
\draw (-6.71,2.71)-- (-5,5);
\draw (-6.71,2.71)-- (-6,5);
\draw (-6,3.41)-- (-5,5);
\draw (-6,3.41)-- (-4,5);
\draw (-5,3.41)-- (-4,5);
\draw (-5,3.41)-- (-6,5);
\draw (-6,3.41)-- (-6.71,1.71);
\draw (-6,3.41)-- (-6,1);
\draw (-6,3.41)-- (-5,1);
\draw (-6,3.41)-- (-4.29,1.71);
\draw (-6,3.41)-- (-4.29,2.71);
\draw (-5,3.41)-- (-4.29,1.71);
\draw (-5,3.41)-- (-5,1);
\draw (-5,3.41)-- (-6,1);
\draw (-5,3.41)-- (-6.71,1.71);
\draw (-5,3.41)-- (-6.71,2.71);
\draw (-4.29,2.71)-- (-5,1);
\draw (-4.29,2.71)-- (-6,1);
\draw (-4.29,2.71)-- (-6.71,1.71);
\draw (-4.29,2.71)-- (-6.71,2.71);
\draw (-4.29,1.71)-- (-6,1);
\draw (-4.29,1.71)-- (-6.71,1.71);
\draw (-4.29,1.71)-- (-6.71,2.71);
\draw (-5,1)-- (-6.71,1.71);
\draw (-5,1)-- (-6.71,2.71);
\draw (-6,1)-- (-6.71,2.71);
\draw (-5.82,0.75) node[anchor=north west] {$\Lambda$};
\draw (-4.29,2.71)-- (-6,5);
\draw (-4.29,1.71)-- (-5,5);
\draw (-5,1)-- (-4,5);
\begin{scriptsize}
\fill [color=black] (-6,1) circle (1.5pt);
\draw[color=black] (-6.25,1.02) node {$v_1$};
\fill [color=black] (-5,1) circle (1.5pt);
\draw[color=black] (-4.7,1.02) node {$v_2$};
\fill [color=black] (-4.29,1.71) circle (1.5pt);
\draw[color=black] (-4.1,1.87) node {$v_3$};
\fill [color=black] (-4.29,2.71) circle (1.5pt);
\draw[color=black] (-4.21,2.87) node {$v_4$};
\fill [color=black] (-5,3.41) circle (1.5pt);
\draw[color=black] (-5.16,3.53) node {$v_5$};
\fill [color=black] (-6,3.41) circle (1.5pt);
\draw[color=black] (-6.05,3.6) node {$v_6$};
\fill [color=black] (-6.71,2.71) circle (1.5pt);
\draw[color=black] (-6.8,2.91) node {$v_7$};
\fill [color=black] (-6.71,1.71) circle (1.5pt);
\draw[color=black] (-6.89,1.85) node {$v_8$};
\fill [color=black] (-4,5) circle (1.5pt);
\draw[color=black] (-3.94,5.17) node {$u_1$};
\fill [color=black] (-5,5) circle (1.5pt);
\draw[color=black] (-4.93,5.17) node {$u_2$};
\fill [color=black] (-6,5) circle (1.5pt);
\draw[color=black] (-6,5.23) node {$u_3$};
\end{scriptsize}
\end{tikzpicture}
\caption{The graph $\Lambda$.}
\label{nnfig0}
\end{center}
\end{figure}
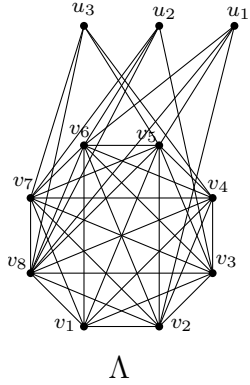

Let $\Upsilon$ be the set of graphs that can be obtained from $\Lambda$ by adding edges between the vertices $u_1$, $u_2$, and $u_3$. Consider the sets $V'=\{v_i:i\in[8]\}$ and $U = \{u_1, u_2, u_3\}$. By examining the structure of $\Lambda$, we observe that for any $i, j\in [8]$, $i\ne j$, and any graph $H\in \Upsilon$, we have $N_H(v_i) \cap U\neq\,N_H(v_j) \cap U$. Therefore, $U = V(H) - V'$ forms a local resolving set of cardinality $n(H) - 8$. This implies that $\dim_l(H) = n(H) - 8$.
Consequently, any connected graph $G$ with $\omega(G) = n(G) - 3$ that has an induced subgraph isomorphic to one of the elements in $\Upsilon$ will also have a local metric dimension of cardinality $n(G) - 8$, as desired.
\end{proof}

To characterize graphs $G$ with  $\dim_l(G) = n(G)-3$, we need the following two specific graphs. Let $\Gamma_1$ be a graph with vertex set $V(F)=\{v_i:\ i\in [6]\}$ and edge set $E(F)=\{v_iv_j:i,j\in\{1,2,3,4\}\}\cup\{v_1v_5,v_1v_6,v_2v_5,v_3v_6\}$. The graph $\Gamma_2$ is obtained from $\Gamma_1$ by adding the edge $v_5v_6$. see Fig.~\ref{fig1}. A graph $G$ is called a $\{\Gamma_1,\Gamma_2\}$-free if it has no induced subgraph isomorphic to either $\Gamma_1$ or $\Gamma_2$.

\begin{figure}[ht!]
\begin{center}
\begin{tikzpicture}
 [
thick,
acteur/.style={circle,fill=black,thick,inner sep=2pt,minimum size=0.1cm} ]
 		\node  (0) at (-2.25, -0.75) {$\Gamma_1$};
		\node  (1) at (-3, 1.5) [acteur,label={[label distance=-1mm]}]{};
		\node  (2) at (-3, -0) [acteur,label={[label distance=-1mm]}]{};
		\node  (3) at (-1.5, -0) [acteur,label={[label distance=-1mm]}]{};
		\node  (4) at (-1.5, 1.5)[acteur,label={[label distance=-1mm]}] {};
		\node  (5) at (-2.25, 2.25) [acteur,label={[label distance=-1mm]}]{};
		\node  (6) at (-0.5, 1.5) [acteur,label={[label distance=-1mm]}]{};
		\node  (7) at (-3.25, -0.25) {$v_4$};
		\node  (8) at (-3.25, 1.2) {$v_2$};
		\node  (9) at (-2.25, 2.5) {$v_5$};
		\node  (10) at (-1.25, 1.2) {$v_1$};
		\node  (11) at (-0.15, 1.5) {$v_6$};
		\node  (12) at (-1.25, -0.25) {$v_3$};
		\node  (13) at (1.5, 1.5) [acteur,label={[label distance=-1mm]}]{};
		\node  (14) at (1.25, 1.2) {$v_2$};
		\node  (15) at (1.25, -0.25) {$v_4$};
		\node  (16) at (4.3, 1.5) {$v_6$};
		\node  (17) at (4, 1.5) [acteur,label={[label distance=-1mm]}]{};
		\node  (18) at (3, -0) [acteur,label={[label distance=-1mm]}]{};
		\node  (19) at (2.25, -0.75) {$\Gamma_2$};
		\node  (20) at (1.5, -0) [acteur,label={[label distance=-1mm]}]{};
		\node  (21) at (3.25, 1.2) {$v_1$};
		\node  (22) at (2.25, 2.5) {$v_5$};
		\node  (23) at (3, 1.5) [acteur,label={[label distance=-1mm]}]{};
		\node  (24) at (2.25, 2.25) [acteur,label={[label distance=-1mm]}]{};
		\node  (25) at (3.25, -0.25) {$v_3$};
 		\draw (1) to (2);
		\draw (2) to (3);
		\draw (3) to (4);
		\draw (4) to (1);
		\draw (1) to (3);
		\draw (4) to (2);
		\draw (5) to (1);
		\draw (5) to (4);
		\draw (4) to (6);
		\draw (6) to (3);
		\draw (13) to (20);
		\draw (20) to (18);
		\draw (18) to (23);
		\draw (23) to (13);
		\draw (13) to (18);
		\draw (23) to (20);
		\draw (24) to (13);
		\draw (24) to (23);
		\draw (23) to (17);
		\draw (17) to (18);
 \draw (24) to (17);
\end{tikzpicture}
\caption{ The graphs $\Gamma_1$ and $\Gamma_2$.}
\label{fig1}
\end{center}
\end{figure}
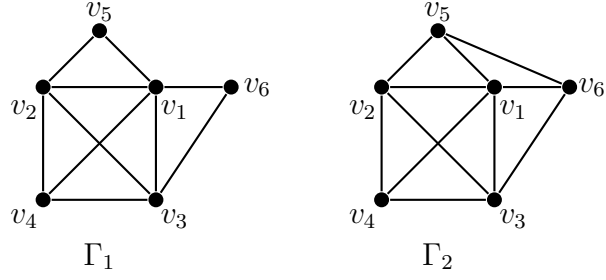

\begin{theorem}\label{th2}
If $G$ is a graph with $n(G) \geq 5$, then $\dim_l(G) = n(G)-3$ if and only if one of the following holds: 
\begin{enumerate}
\item[(i)] $\omega(G)=n(G)-2$ and $G$ is a $\{\Gamma_1,\Gamma_2\}$-free,
\item[(ii)] $G\cong C_5$, 
\item[(iii)] $G \cong K_{n(G)}^-(\lambda,\mu)$, $\lambda \geq \mu \geq 2$, $\lambda+\mu < n(G)$.
\end{enumerate}
\end{theorem}

\begin{proof}
Let $G$ be a graph with $n = n(G) \geq5$. If $ \omega(G) \leq n - 3$, then by employing Theorem \ref{th1}, we have $\dim_l(G) = n - 3$ if and only if one of (ii) and (iii) is fulfilled. To complete the proof, it is sufficient to show that if $\omega(G) \geq n - 2$, then $\dim_l(G) = n - 3$ if and only if $\omega(G) = n - 2$ and $G$ is $\{\Gamma_1, \Gamma_2\}$-free. 
We will prove this in two parts.

Assume first that $\omega(G) = n-2$ and $G$ is $\{\Gamma_1, \Gamma_2\}$-free. Let $Q$ a largest clique of $G$, and let its vertices be $w_1, \ldots, w_{n-2}$. First, we show that $\dim_l(G) \geq n-3$. It is enough to prove that for any subset $X$ of $V(G)$ with $|X|=4$, there exist adjacent vertices $x_1$ and $x_2$ in $X$ such that $d_G(x_1,a)=d_G(x_2,a)$ for each  $a\in(V(G)-X)$. To prove this, observe first that $|X\cap Q|\in\{2,3,4\}$. If $|X\cap Q|=2$ and $X\cap Q=\{w_i,w_j\}$ for some $i,j\in [n-2]$, then $w_i$ and $w_j$ satisfy the condition. If $|X\cap Q|=3$ and $X\cap Q=\{w_i,w_j,w_k\}$ for some $i,j,k\in [n-2]$, then $G[\{w_i,w_j,w_k\}]\cong K_3$. In this case we may without loss of generality assume that $X\cap \{u,v\}=u$. Then at least two members of $\{w_i,w_j,w_k\}$ are either adjacent or non-adjacent with $v$, thus at least one pair of $w_i$, $w_j$, and $w_k$ satisfies the conditions. If $|X\cap Q|=4$ and $X\cap Q=\{w_i,w_j,w_k,w_l\}$ for some $i,j,k,l\in [n-2]$, then $G[\{w_i,w_j,w_k,w_l\}]\cong K_4$. In this case, since $G$ is $\{\Gamma_1, \Gamma_2\}$-free, we infer that at least one pair of $w_i$, $w_j$, $w_k$, and $w_l$ satisfies the conditions. So, $\dim_l(G)\geq n-3$.

We next show that $\dim_l(G)\leq n-3$. Let $V(G)-Q=\{u,v\}$ and consider two cases. Assume first that $uv\notin E(G)$. In this case, since $\omega(G)=n-2$, for any $i\in [n-2]$, both $(N_G(w_i)\cap Q)-N_G(u)\neq\emptyset$ when $uw_i \in E(G)$ and $(N_G(w_i)\cap Q)-N_G(v)\neq\emptyset$ when $vw_i\in E(G)$ hold. Therefore, for any $i\in [n-2]$, the set $V(G)-\{u,v,w_i\}$ is a local resolving set for $G$. So, $\dim_l(G)\leq n-3$. Assume second that $uv\in E(G)$. In this case, based on the explanations from the first case, if $x$ is a member of either $(N_G(u)\cap Q)-N_G(v)$ or $(N_G(v)\cap Q)-N_G(u)$, then for some $i\in [n-2]$ with $w_i\neq x$, the set $V(G)-\{u,v,w_i\}$ is a local resolving set for $G$. So, suppose $X=N_G(u)\cap Q=N_G(v)\cap Q$. Since $\omega(G)=n-2$, $|Q-X|\geq 2$. Thus, if $x\in X$ and $w\in (Q-X)$, then one can observe that both $V(G)-\{x,w,u\}$ and $V(G)-\{x,w,v\}$ are local resolving sets for $G$. Thus, in case (b), the inequality $\dim_l(G)\leq n-3$ also holds.

\medskip
We have thus proved that if $\omega(G)=n-2$ and $G$ is a $\{\Gamma_1, \Gamma_2\}$-free, then $\dim_l(G)=n-3$. Conversely, assume that $\dim_l(G)=n-3$. Then Theorem \ref{cth0} implies that $\omega(G) \le n-2$. If $\omega(G) \le n-3$, then by Theorem~\ref{th1} we get that $\dim_l(G)=n-3$ if and only if  $G$ is either $C_5$ or $K_n^{-}(\lambda, \mu)$, where $\lambda \geq \mu \geq 2$ and $\lambda+\mu<n$. To complete the argument we thus need to prove that if $\omega(G)=n-2$ and $\dim_l(G)=n-3$, then $G$ is $\{\Gamma_1, \Gamma_2\}$-free.

To demonstrate this, let's assume, for the sake of contradiction, that $\omega(G)=n-2$, $\dim_l(G)=n-3$, and that $G$ contains an induced subgraph $H$ that is isomorphic to either $\Gamma_1$ or $\Gamma_2$. 
Let $S$ be the set of vertices of $H$ corresponding to the vertices $v_1$, $v_2$, $v_3$, and $v_4$. In this scenario, by examining the structure of $\Gamma_1$ and $\Gamma_2$, we can conclude that $V(G) - S$ serves as a local resolving set for $G$. This indicates that $\dim_l(G) \leq n - 4$, which contradicts our earlier assumption that $\dim_l(G) = n - 3$. 
Therefore, if $\omega(G) = n - 2$ and $\dim_l(G) = n - 3$, it follows that $G$ must be $\{\Gamma_1, \Gamma_2\}$-free, as required.
\end{proof}

\begin{corollary}
If $G$ is a graph with $n(G)\ge 5$ and $\omega(G)=n(G)-2$, then
$$n(G)-4\leq\dim_l(G)\leq n(G)-3\,,$$
where the right equality holds if and only if $G$ is $\{\Gamma_1,\Gamma_2\}$-free, and the the left equality holds if and only if $G$ contains an induced subgraph isomorphic to $\Gamma_1$ or $\Gamma_2$.
\end{corollary}

\begin{proof}
Let $G$ is a graph with $n(G)\ge 5$ and $\omega(G) = n(G) - 2$.  Theorem~\ref{cth0} implies that $\dim_l(G)\leq n(G)-3$, while Theorem~\ref{thm:Okamoto-2} gives $\dim_l(G) \ge n(G) - 4$. 
Moreover, Theorem~\ref{th2} yields that $\dim_l(G) = n(G)-3$ holds if and only if $G$ is $\{\Gamma_1,\Gamma_2\}$-free. Consequently, $\dim_l(G) = n(G)-4$ holds if and only if $G$ contains an induced subgraph isomorphic to $\Gamma_1$ or $\Gamma_2$.
\end{proof}

\section{On a conjecture and a problem}
\label{sec:conjectures}

In this section, we focus on a conjecture and a problem from two earlier papers which relate the local metric dimension and the clique number. We  start with the following: 

\begin{conjecture} {\rm \cite[Conjecture 2]{Ghalavand1}}
\label{con2}
If $G$ is a graph with $n(G) \geq \omega(G)+1 \geq 4$, then 
\[\dim_l(G)\leq\left(\frac{\omega(G)-2}{\omega(G)-1}\right)n(G)\,.\]
\end{conjecture}

It is demonstrated in~\cite{Ghalavand1} that if Conjecture~\ref{con2} is true, then the bound is asymptotically best possible. By Theorem~\ref{cth0}, Conjecture~\ref{con2} holds for graphs with clique number $\omega(G) = n(G)-1$, as $\left( \frac{n(G)-3}{n(G)-2} \right)n(G) - (n(G)-2) = \frac{n(G)-4}{n(G)-2} \geq 0$, where $n(G) \geq 4$. The following corollary greatly extends this results. 

\begin{corollary}
If $G$ is a graph with $\omega(G)\in\{n(G)-1,n(G)-2,n(G)-3\}$ and $n(G)\geq \omega(G)+1 \geq 4$, then 
\[\dim_l(G)\leq\left(\frac{\omega(G)-2}{\omega(G)-1}\right)n(G).\]
\end{corollary}

\begin{proof}
Set $n=n(G)$ and $k=\omega(G)$. Since $n\geq\,k+1\geq4$, it follows that $n\geq4$. If  $n=4$, then $k=n-1$, and if $n=5$, then $k=n-1$ or $n-2$. Therefore, if $n\geq\,k+1\geq4$, then $\left(\frac{k-2}{k-1}\right)n-(k-1)\geq0$ for $k\in\{n-1,n-2\}$ and   $\left(\frac{k-2}{k-1}\right)n-k\geq0$ for $k=n-3$. Accordingly, based on Theorems \ref{cth0}, \ref{th1}, and \ref{th2}, we can conclude that $\dim_l(G)\leq\left(\frac{k-2}{k-1}\right)n$ for $k\in\{n-1,n-2,n-3\}$ and $n\geq\,k+1\geq4$.
\end{proof}

We now turn our attention to the following:

\begin{problem} {\rm  \cite[Problem 1]{Abrishami1}} 
\label{con1}
If $G$ is a planar graph with $n(G) \geq 2$, is it then true that  
\[\dim_l(G) \leq \left\lceil\frac{n(G)+1}{2}\right\rceil\,?\]
\end{problem}

By Theorem~\ref{cth1}, Problem~\ref{con1} has a positive answer for triangle-free planar graphs, that is, for planar graphs $G$ with $\omega(G) = 2$. In the following we present an example which demonstrates that in general the problem has a negative answer.

If $\ell\ge 2$ is an integer, then let $G_\ell$ be the graph obtained from the disjoint union of $\ell$ complete graphs $K_3$, by adding a new vertex and make it adjacent to all vertices in the $\ell$ copies of $K_3$. Then $G_\ell$ is a planar graph. Since the vertices of every $K_3$ form a true twin equivalence class, at least two of them must lie in every local resolving set. Consequently, $\dim_\ell(G)\ge 2\ell$. In addition, a set consisting of two vertices from each of the $K_3$ is a local resolving set, so we have $\dim_\ell(G) = 2\ell$. We can conclude that Problem~\ref{con1}  has a negative answer in general.

Note that $\omega(G_\ell) = 4$. It remains open whether Problem~\ref{con1} has a positive answer for planar graphs $G$ with $\omega(G) = 3$. 

\section*{Acknowledgments}


The research of Ali Ghalavand and Xueliang Li was supported the by NSFC No.\ 12131013 and 12161141006. Sandi Klav\v{z}ar was supported by the Slovenian Research Agency (ARIS) under the grants P1-0297, N1-0355, and N1-0285.

\section*{Conflicts of interest} 

The authors declare no conflict of interest.

\section*{Data availability} 

No data was used in this investigation.

\end{document}